\documentclass{amsart}
\usepackage{amsthm}
\usepackage{amsmath,amssymb}
\usepackage{mathrsfs}
\usepackage[foot]{amsaddr}
\usepackage{setspace}

\usepackage[colorlinks=true,linkcolor=blue,citecolor=red]{hyperref}
\usepackage[nameinlink]{cleveref}

\theoremstyle{plain}
\newtheorem{theorem}{Theorem}
\newtheorem*{theorem*}{Theorem}
\newtheorem{lemma}[theorem]{Lemma}

\newtheorem{corollary}[theorem]{Corollary}

\theoremstyle{definition}

\theoremstyle{remark}
\newtheorem{remark}[theorem]{Remark}

\newcommand{\RR}{\mathbb{R}}

\newcommand{\D}{\Delta}

\newcommand{\Om}{\Omega}

\newcommand{\I}{\int\limits_}
\allowdisplaybreaks

\begin{document}

	\title[Improved Hardy inequalities on Riemannian manifolds]
	{Improved Hardy inequalities on Riemannian Manifolds}
	\author{Kaushik Mohanta$^1$*}
	\email{$^1$kaushik.k.mohanta@jyu.fi}
	\address{$^1$University of Jyv\"askyl\"a, Finland}

	\author{Jagmohan Tyagi$^2$}
	\email{$^2$jtyagi@iitgn.ac.in}

	\address{$^2$Indian Institute of Technology Gandhinagar, India}
	\begin{abstract}
		We study the following version of Hardy-type inequality on a domain $\Om$ in a Riemannian manifold $(M,g)$:
		\begin{equation*}
			\I{\Om}|\nabla u|_g^p\rho^\alpha dV_g
			\geq \left(\frac{|p-1+\beta|}{p}\right)^p\I{\Om}\frac{|u|^p|\nabla \rho|_g^p}{|\rho|^p}\rho^\alpha dV_g\\
			+\I{\Om} V|u|^p\rho^\alpha dV_g,
			\quad \forall\ u\in C_c^\infty (\Om).
		\end{equation*}
		We provide sufficient conditions on $p, \alpha, \beta,\rho$ and $V$ for which the above inequality holds. This generalizes earlier well-known works on Hardy inequalities on Riemannian manifolds. The functional setup covers a wide variety of particular cases, which are discussed briefly: for example, $\RR^N$ with $p<N$, $\RR^N\setminus \{0\}$ with $p\geq N$, $\mathbb{H}^N$, etc. 
	\end{abstract}
	
	\maketitle
	
	
	\section{Introduction}
	
	The Hardy inequality plays an important role in analysis and in the theory of partial differential equations. In $\RR^N,$ it reads as follows:
	\begin{equation}\label{hardy}
	\I{\RR^N}\frac{|u|^p}{|x|^p}\leq c \I{\RR^N} |\nabla u|^p, \quad \forall\ u\in C_c^\infty(\RR^N),
	\end{equation}
	where the constant $c$ is independent of $u.$ When we restrict $u$ to be in the space $W_0^{1, p}(\RR^N)$, with $p<N,$ or in $W_0^{1, p}(\RR^N\setminus \{0\})$, with $p\geq N$, \eqref{hardy} holds; in that case the smallest possible choice for $c$ becomes $\left(\frac{p}{N-p}\right)^p,$ and this constant is never achieved (see for example \cite{BrVa,AdChRa}). Hence there is a scope to get an improvement in the above inequality. Brezis-V\'azquez \cite{BrVa} have shown for the case $p=2$, that we can add an $L^2$-term in the left hand side of \eqref{hardy} even with the best constant. This was further generalized in many directions, for example, see \cite{AdChRa,AdSe,AdEs,AbCoPe,BaFiTe,Chau,CuPe,AdFiTe,TeZo,DuLaPh,Nguy,DuLaTrYi}.
	
	There are many results in literature on this subject in the context of a complete Riemannian manifolds $(M, g).$ An important result in this direction is due to Kombe-O\"zaydin \cite{KoOz}. They proved that for a nonnegative function $\rho$ with $|\nabla \rho|_g=1$ and $\D\rho\geq \frac{C}{\rho}$, the following holds
	\begin{equation*}
		\left(\frac{C+1+\alpha -p}{p}\right)^p \I{M}\frac{|u|^p \rho^\alpha}{\rho^p}\leq  \I{M} |\nabla u|_g^p\rho^\alpha \quad \mbox{ for any } u\in C_c^\infty (M).
	\end{equation*}
	Moreover, for a bounded domain $\Om$ with smooth boundary and in the case $p=2$, they also proved that the above inequality still holds for any $u$ in $C_c^\infty(\Om)$, if we add a reminder term of the form $C_1\left(\I{\Om}|\nabla u|^q\rho^{q\alpha/2}\right)^{2/q}$ in the left hand side of the equation, where $1<q<2$ and $C=C(N,q,\Om)$. D'Ambrosio-Dipierro \cite{AmDi} proved another version of Hardy inequality, where the restriction $|\nabla \rho|=1$ is not there. They showed that if $\Om$ is an open set in $M$ and there is a $\rho:\Om\to [0,\infty)$ such that $\rho \in W^{1,p}_{loc}(\Om)$ with $\D_p \rho\leq 0$ weakly, then $\frac{|\nabla \rho|}{\rho} \in L^p_{loc}$, and for any $u\in C_c^\infty (\Om)$, the following holds:
	\begin{equation*}
		\left(\frac{p-1}{p}\right)^p \I{\Om}\frac{|u|^p|\nabla \rho|_g^p}{\rho^p}\leq  \I{\Om} |\nabla u|_g^p.
	\end{equation*}
	They discussed, in details, the advantage and applications of this kind of Hardy inequality.
	For other related results, the reader may refer to \cite{KoAb,Thiam,Kristaly,MeWaZh,DeFrPi,AhKo} and the references therein.
	As of now, there is no known result regarding Hardy inequality with a reminder term as prescribed by D'Ambrosio-Dipierro \cite{AmDi}. We wish to address this problem in this paper by proving a slightly more general version of the result. 
	
	An interesting feature of our method is that the results of \cite{AmDi} and \cite{KoOz} follow immediately. Also the proof becomes much simpler (see \Cref{hardy1}) provided no reminder term is expected. The formulation allows greater control over the constant, and, at least in the Euclidean case, we can get the inequality with the best constants. Here we mainly focus on the improved Hardy inequality in all its generality. We refrain ourselves from studying the special cases exclusively, although that is very much possible to use our results to achieve improvement of the inequalities obtained in \cite{AmDi}.
	
	Throughout the article, $(M,g)$ stands for a fixed oriented Riemannian manifold. We shall often use the notation $\left<X,Y\right>$ to denote $g(X,Y)$ for any two vector fields $X$ and $Y$. We use the symbol $dV_g$ to denote the volume form, however, the symbol will be often dropped when there is no scope for confusion.
	
	Now, we state  the main results of this paper, which we shall prove in Section 2.  The first one is the following Hardy inequality, which, in the particular case $\alpha=-\beta$, is proven in \cite{AmDi}. However, our proof is much more simpler.
	\begin{theorem}\label{hardy1}
		Let $1< p<\infty$, $\alpha,\beta\in\RR$, $M$ be a complete Riemannian manifold, $\Om$ be a domain in $M$ with boundary $\partial \Om$ (possibly empty). Let there exist a function $\rho: \Om\to (0,\infty)$ with $|\nabla \rho|_g\rho^\frac{-p+\alpha}{p}\in L^p_{loc}(\Om)$ such that
		\begin{equation}\label{suphar-weak}
			\frac{1}{p-1+\beta} \I{\Om}\left<\nabla \xi,|\nabla \rho|^{p-2}\nabla \rho \right>\rho^{-p+1+\alpha}
			\geq \I{\Om}\frac{|\nabla \rho|^p}{\rho^{p-\alpha}}\xi,
			\quad \forall \ \xi\in C_c^1(\Om)\mbox{ with } \xi\geq 0.
		\end{equation}
		Then, for any $u\in W^{1,p}_0(\Om)$, we have the following inequality:
		\begin{equation*}
			\I{\Om}|\nabla u|_g^p\rho^\alpha dV_g
			\geq \left(\frac{|p-1+\beta|}{p}\right)^p\I{\Om}\frac{|u|^p|\nabla \rho|_g^p}{\rho^p}\rho^\alpha dV_g.
		\end{equation*}
	\end{theorem}
	\begin{theorem}\label{main}
		Let $p,\alpha,\beta,M,\Om$ and $\rho$ be as in \Cref{hardy1}. Let the constant $C=C(p)$ be as in \Cref{p-estimate} for $p\geq 2$ and in the case $p<2$, $C(p):=2^{p-3}p(p-1)$. 
		Further, in the case $p<2$, assume that for any $\xi\in C_c^1(\Om)\mbox{ with } \xi\geq 0$
		\begin{equation}\label{p<2}
		\begin{cases}
		\rho^\frac{-p+1-\beta}{p}\in W^{1,p}_{loc}(\Om),\\
		(p-1+\beta) \I{\Om}\left<\nabla \xi,|\nabla \rho|^{p-2}\nabla \rho \right>\rho^{\alpha+\beta}
		\geq \I{\Om}\frac{|\nabla \rho|^p}{\rho}\rho^{\alpha+\beta}\xi.
		\end{cases}
		\end{equation}
		If there exist functions
		$V:\Om\to [0,\infty)$ and $\varphi\in W^{1,p}_{loc}(\Om)$ such that for any $\xi\in C_c^1(\Om)$,
		\begin{multline}\label{ode-weak}
			\I{\Om}\left<|\nabla \varphi|^{p-2}\nabla \varphi,\nabla\xi\right>\varphi^{-p+1}\rho^\alpha
			- (p-1) \I{\Om}|\nabla \varphi|^p\xi\rho^\alpha\varphi^{-p}
			- (p-1+\beta) \I{\Om}\left<|\nabla \varphi|^{p-2}\nabla \varphi,\nabla\rho\right>\xi\varphi^{-p+1}\rho^{-1+\alpha}\\
			\geq \I{\Om} \frac{V}{C(p)}\xi \rho^\alpha,
		\end{multline}
		then for any $u\in W^{1,p}_0(\Om)$, we have the following inequality:
		\begin{equation}\label{eq-hardy}
		\I{\Om}|\nabla u|_g^p\rho^\alpha dV_g
		\geq \left(\frac{|p-1+\beta|}{p}\right)^p\I{\Om}\frac{|u|^p|\nabla \rho|_g^p}{|\rho|^p}\rho^\alpha dV_g+\I{\Om} V|u|^p\rho^\alpha dV_g.
		\end{equation}
	\end{theorem}
	
	Before moving further, let us observe that the hypotheses \eqref{suphar-weak}, \eqref{p<2}, \eqref{ode-weak} can be thought of as weak formulations of certain problems; we explain this in the following using the Green's theorem (see \Cref{green}). 
	Consider the following condition on $\rho$: 
	\begin{equation}\label{eq5}
	\frac{-1}{p-1+\beta}\D_p\rho
	\geq\frac{\alpha+\beta}{p-1+\beta}\frac{|\nabla \rho|^p}{\rho}.
	\end{equation}
	In the special case $\alpha=-\beta$, this is just the $p$-superharmonicity condition. We rewrite this as 
	$$
	\frac{-1}{p-1+\beta}\rho^{-p+1+\alpha}\D_p\rho
	+\frac{p-1-\alpha}{p-1+\beta}\rho^{-p+\alpha} |\nabla \rho|^p
	\geq \frac{|\nabla \rho|^p}{\rho^{p-\alpha}},
	$$
	and then multiply both sides by a test function $\xi\in C_c^1(\Om)$, integrate over $\Om$, and then apply the product rule of divergence operator to get
	$$
	\frac{-1}{p-1+\beta} \I{\Om}\mbox{div}\left(\rho^{-p+1+\alpha} |\nabla \rho|^{p-2}\nabla \rho \right)\xi
	\geq \I{\Om}\frac{|\nabla \rho|^p}{\rho^{p-\alpha}}\xi.
	$$
	An application of Green's theorem then gives \eqref{suphar-weak}. Thus \eqref{suphar-weak} is an weak formulation of \eqref{eq5}. Similarly, the second condition of \eqref{p<2} can be interpreted as the weak formulation of 
	\begin{equation*}
		\frac{-1}{p-1+\beta}\D_p\rho
		\geq\frac{p-1+\alpha+2\beta}{p-1+\beta}\frac{|\nabla \rho|^p}{\rho},
	\end{equation*}
	and \eqref{ode-weak} can be regarded as a weak formulation of
	$$
	\D_p\varphi+(p-1+\alpha+\beta)\left<|\nabla \varphi|^{p-2}\nabla \varphi, \frac{\nabla\rho}{\rho}\right>+\frac{V}{C(p)}\varphi^{p-1}
	\leq 0.
	$$
	
	Now, we discuss some immediate consequences of \Cref{main}. The following results says that if $\left(\frac{|p-1+\beta|}{p}\right)^p$ is the best constant in \eqref{eq-hardy}, then  \eqref{ode-weak} has no solution when $V=\frac{|\nabla \rho|^p}{\rho^p}$. 
	\begin{corollary}\label{cor1}
		Let $p,\alpha,\beta,\rho$ be as in \Cref{main}, and $V=C\frac{|\nabla \rho|}{\rho}$ be such that there exists some $\varphi$ for which \eqref{ode-weak} is satisfied. Then the constant $\left(\frac{|p-1+\beta|}{p}\right)^p$ in \eqref{eq-hardy} is not sharp. 
	\end{corollary}
	
	\begin{remark}
	The above result says nothing about when the constant $\left(\frac{|p-1+\beta|}{p}\right)^p$ in \eqref{eq-hardy} is sharp. However it is easy to see that the constant is sharp if and only if we can never have $V=C\frac{\nabla \rho}{\rho}$ in \eqref{eq-hardy} for any $C>0$. So the question of whether the constant is sharp is related to necessity of \eqref{ode-weak} in \eqref{eq-hardy}; whereas result \Cref{main} concerns with the sufficiency part. In some very particular case the condition may also be necessary as can be seen from \cite[Theorem~1]{Mora}.
		\end{remark}
	 
	We discuss some special cases of \Cref{main}. In the Euclidean setup, we get the following results as a corollary. Let us take $M=\RR^N$, $\rho=|x|^d$. Then $\nabla \rho= d|x|^{d-2}x$ and we get the following corollary:
	\begin{corollary}\label{cor2}
		Assume $1<p<\infty$, $d\neq 0$, $\alpha,\beta\in\RR$ be such that $p-1+\beta\neq 0$, 
		$$
		\frac{\alpha+\beta}{p-1+\beta}\leq 0 \quad 
		\mbox{ and }\quad \frac{(N+(d-1)(p-1))}{d(p-1+\beta)}\leq 0,
		$$
		and
		$$
		\frac{(N+(d-1)(p-1))}{d(p-1+\beta)}\leq -1 \quad \mbox{ when } p<2.
		$$
		Assume further that there exists a function $V:\RR^N\to \RR$ such that the following problem admits a weak solution $\varphi\in W^{1,p}_{loc}(\Om)$:
		\begin{equation}\label{eq6}
		\D_p \varphi +\frac{d(p-1+\alpha+\beta)}{|x|^2}|\nabla \varphi|^{p-2}\nabla\varphi\cdot x+\frac{V}{C(p)}\varphi^{p-1}\leq 0.
		\end{equation}
		Consider the two cases:\\
		(i) $\Om\subseteq\RR^N$, $p<N+\alpha d$ with 
		$$
		(p-1+\beta)d<N-p \quad \mbox{ when } p<2,
		$$
		and\\
		(ii) $\Om\subseteq\RR^N\setminus \{0\}$, $p\geq N+\alpha d$.\\
		Then, in both cases, we have the following inequality:
		
		\begin{equation*}
		\I{\Om}|\nabla u|_g^p|x|^{\alpha d} dx\\
		\geq \left(\frac{|d(p-1+\beta)|}{p}\right)^p\I{\Om}\frac{|u|^p}{|x|^{p-\alpha d}} dx+\I{\Om} V|u|^p|x|^{\alpha d},\,\, \forall\,u\in W^{1,p}_0(\Om).
		\end{equation*}
		
	\end{corollary}
	Remark that in the case $\alpha=\beta=0,d=-1$ and $\Om=\RR^N$, \cite[Proposition~1.2]{CaKrLa} implies that \eqref{eq6} has no solution in $\RR^N$ for any $1<p<N$.

	Now, let us consider the so-called critical case: $p=N\geq 2$. Set $\Om:=B_1(0)\subset \RR^N$, $\rho=-\log |x|$. We have
	\begin{corollary}\label{cor3}
		Let $N\geq 2$, $B_1(0)$ denote the unit ball in $\RR^N$, $\alpha,\beta\in\RR$ be such that $p-1+\beta\neq 0$ and  
		$$
		\frac{\alpha+\beta}{N-1+\beta}\leq 0.
		$$
		Assume further that there exists a function $V: B_1(0)\subset \RR^N\to \RR$ such that the following problem admits a weak solution $\varphi \in W^{1,p}_{loc}(B_1(0)):$
		$$
		\D_N\varphi -\frac{(N-1+\alpha+\beta)|\nabla\varphi|^{N-2}}{|x|^2\log x}\nabla\varphi\cdot x+\frac{V}{C(N)}\leq 0.
		$$
		Then for any $u\in W^{1,N}_0(B_1(0))$, we have
		\begin{equation*}
			\I{B_1(0)}|\nabla u|^N|\log |x||^\alpha dx
			\geq \left(\frac{|N-1+\beta|}{N}\right)^N\I{B_1(0)}\frac{|u|^N}{|x|^N|\log |x||^{N-\alpha}} dx\\
			+\I{B_1(0)} V|u|^N|\log |x||^{\alpha}dx.
		\end{equation*}
	\end{corollary}
	Consider the case $M=\mathbb{H}^N:=\RR^{N-1}\times(0, \infty)$ and set $\rho(x):=x_N$. This gives $|\nabla\rho|=1$ and $\D_p\rho=0$. This implies
	\begin{corollary}\label{cor4}
		Let $1< p<\infty$, $\alpha,\beta\in\RR$. Let the constant $C=C(p)$ be as in \Cref{main}. Assume
		$$
		(\alpha+\beta)(p-1+\beta)\leq 0.
		$$
		
		Further, in the case $p<2$, assume that 
		$$
		(p-1+\alpha+2\beta)(p-1+\beta) \leq 0.
		$$
		If there exists a positive $p$-harmonic function
		$\varphi\in W^{1,p}_{loc}(\mathbb{H}^N)$ such that $(p-1+\alpha+\beta)\frac{\partial \varphi}{\partial x_N}\leq 0$, 
		then for any $u\in W^{1,p}_0(\mathbb{H}^N)$, we have the following inequality:
		\begin{equation*}
			\I{\mathbb{H}^N}|\nabla u|_g^px_N^\alpha dx
			\geq \left(\frac{|p-1+\beta|}{p}\right)^p\I{\mathbb{H}^N}\frac{|u|^p}{x_N^{p-\alpha}} dx\\
			-C(p)(p-1+\alpha+\beta)\I{\mathbb{H}^N} \frac{\partial \varphi}{\partial x_N}\frac{|\nabla \varphi|^{p-2}}{\varphi^{p-1}}x_N^{\alpha+1}|u|^p dx.
		\end{equation*}
	\end{corollary}
	\smallskip

	\begin{remark}
		The conditions \eqref{suphar-weak} in \Cref{hardy1} and \eqref{ode-weak} in \Cref{main} may seem artificially imposed at first glance. However, in the most commonly used form of Hardy inequality, that is in the setup of \Cref{cor2} with $\alpha=\beta=0$, \eqref{suphar-weak} holds automatically with proper choice of $d$. In case of Riemannian manifolds validity of \Cref{suphar-weak} is not immediate, it is related to $p$-hyperbolicity of the underlying manifold. A discussion regarding this can be found in \cite{AmDi}.\\  		
		In the particular case $p=2$, $\Om=\RR^N$, and when $V$ is radial, \eqref{ode-weak} is actually a necessary condition too for \Cref{eq-hardy} to hold. This can be seen from some symmetrization argument (to reduce the condition to its one-dimensional analogue) and \cite[Theorem~1]{Mora}.
		\end{remark}

	In \Cref{sec-main}, we shall prove some preliminary results, \Cref{hardy1} followed by the proof of \Cref{main}.
	\bigskip
	
	\section{Proof of the theorems}\label{sec-main}
	The proof of the following lemma can be found in \cite[Theorem~III.7.6.]{Chavel} for the case $p=2$. The proof of this version can also be done similarly as the essence of the proof lies in the Stokes theorem and in the product rule of divergence operator:
	$$
	\mbox{div}(fX)=f\mbox{div}(X)+\left<\nabla f, X\right>.
	$$
	\begin{lemma}[Green's Formula]\label{green}
		Let $p>1$, $M$ be complete, oriented Riemannian manifold, $\Om$ a domain in
		$M$ with smooth boundary. Let  $f\in C^2 (\Om)$, $\xi\in C_c^1 (\Om)$. Then, we have
		$$
		\I{\Om}<|\nabla f|^{p-2}\nabla f,\nabla \xi>=-\I{\Om}\xi\D_p f.
		$$
	\end{lemma}
	The following result plays a key role in the proof of the theorem.
	\begin{lemma}\label{p-estimate}
		Let $x\in M$ and $X_x,Y_x\in T_xM$ be two tangent vectors. Then, for $p\geq 2$, there is a constant $C=C(p)>0$ such that
		$$
		|X_x+Y_x|_g^p-|X_x|_g^p \geq
		C(p)|Y_x|_g^p
		+ p|X_x|_g^{p-2}<X_x,Y_x>,
		$$
		and for $1\leq p \leq 2$,
		$$
		|X_x+Y_x|_g^p-|X_x|_g^p
		\geq \frac{p(p-1)}{2}\frac{|Y_x|^2}{||X_x|+|Y_x||^{2-p}}+p|X_x|^{p-2}<X_x,Y_x>.
		$$
		In general, for $1<p<\infty$, we have
		$$
		|X_x+Y_x|_g^p-|X_x|_g^p
		\geq p|X_x|^{p-2}<X_x,Y_x>.
		$$
	\end{lemma}
	\begin{proof}
		For the case $p\geq 2$, refer to \cite[Chapter~12]{Lindq} and recall that any $N$-dimensional Hilbert space is isomorphic to $\RR^N$. We give a proof of the case $1\leq p \leq 2$, which is adapted from \cite[Lemma~1]{GaGrMi}. 
		
		In the following calculation, we omit the point $x$ for better readability. Consider the function
		$$
		f(t):=|X+tY|^p.
		$$
		It can be easily verified that, for $t\in (0,1)$, $f''(t)\geq p(p-1)|Y|^2|X+tY|^{p-2}$.
		Using this, Taylor's theorem, the fact that $p-2\leq 0$, we get
		\begin{align*}
			|X+Y|^p=\ &f(1)\\
			=\ &f(0)+f'(0)+\I{0}^1(1-t)f''(t)dt\\
			\geq\ & |X|^p+p|X|^{p-2}\left<X,Y\right>\\
			&+p(p-1)\I{0}^1(1-t)|Y|^2|X+tY|^{p-2}dt\\
			\geq\ & |X|^p+p|X|^{p-2}\left<X,Y\right>+\frac{p(p-1)}{2}|Y|^2||X|+|Y||^{p-2}.
		\end{align*}
		This proves the lemma.	
	\end{proof}
	
	We now present the proof of our first main result.
	\begin{proof}[\textbf{Proof of \Cref{hardy1}}]
		By a density argument, we can always assume that  $u\in C_c^1(\Om)$. Set $w(x)=u(x)\rho^\frac{-p+1-\beta}{p}(x)$ in $\Om$. Then 
		$$
		\nabla u=\rho^\frac{p-1+\beta}{p}\nabla w + \frac{p-1+\beta}{p}\rho^\frac{-1+\beta}{p}w\nabla \rho.
		$$
		Note that if $\I{\Om}|\nabla u(x)|_g^p\rho^\alpha dV_g=\infty$, then we have nothing to prove, so we assume it to be finite. In the following calculation, we need the term $\frac{|u|^p|\nabla \rho|_g^p}{\rho^{p-\alpha}}$ to be integrable over $\Om$; this is indeed true, as $u\in C_c^1(\Om)$ and $\frac{|\nabla \rho|_g^p}{\rho^{p-\alpha}}\in L^1_{loc}(\Om)$ according to hypothesis. Using \Cref{p-estimate} and \eqref{suphar-weak}, we have
		\begin{align*}
			\I{\Om}&|\nabla u(x)|_g^p\rho^\alpha dV_g
			- \left(\frac{|p-1+\beta|}{p}\right)^p\I{\Om}\frac{|u(x)|^p|\nabla \rho(x)|_g^p}{|\rho(x)|^p}\rho^\alpha dV_g\\
			=\ & \I{\Om} \left(\left|\rho^\frac{p-1+\beta}{p}\nabla w + \frac{p-1+\beta}{p}\frac{w\nabla\rho}{\rho^\frac{1-\beta}{p}}  \right|^p_g-\left(\frac{|p-1+\beta|}{p}\right)^p\left|\frac{w\nabla \rho}{\rho^\frac{1-\beta}{p}}\right|_g^p\right)\rho^\alpha\\
			\geq\ & p\left(\frac{|p-1+\beta|}{p}\right)^{p-2}\left(\frac{p-1+\beta}{p}\right)\I{\Om}\left|\frac{w\nabla \rho}{\rho^\frac{1-\beta}{p}}\right|^{p-2}\left< \rho^\frac{p-1+\beta}{p}\nabla w,\frac{w\nabla \rho}{\rho^\frac{1-\beta}{p}} \right>\rho^\alpha\\
			=\ & \left(\frac{|p-1+\beta|}{p}\right)^{p-2}\left(\frac{p-1+\beta}{p}\right) \I{\Om}\left<\nabla |w|^p,|\nabla \rho|^{p-2}\nabla \rho \right>\rho^{\alpha+\beta}\\
			=\ & \left(\frac{|p-1+\beta|}{p}\right)^{p-2}\left(\frac{p-1+\beta}{p}\right) \I{\Om}\left<\nabla (\rho^{-p+1-\beta}|u|^p),|\nabla \rho|^{p-2}\nabla \rho \right>\rho^{\alpha+\beta}\\
			=\ & \left(\frac{|p-1+\beta|}{p}\right)^{p-2}\left(\frac{p-1+\beta}{p}\right) \I{\Om}\left<\nabla |u|^p,|\nabla \rho|^{p-2}\nabla \rho \right>\rho^{-p+1+\alpha}\\
			&- \left(\frac{|p-1+\beta|}{p}\right)^p p \I{\Om}\frac{|\nabla \rho|^p}{\rho^{p-\alpha}}|u|^p\\
			\geq\ & 0.
		\end{align*}
		In the last line of the above calculation, we have used the hypothesis by using the fact that $|u|^p\in C_c^1(\Om)$ as $p>1$.
	\end{proof}
	
	Next, we proceed with the proof of our second main result.
	\begin{proof}[{\textbf{Proof of \Cref{main}}}]	
		As in the proof of \Cref{hardy1}, we shall assume $u\in C_c^1(\Om)$ and we set $u(x)=w(x)\rho^\frac{p-1+\beta}{p}(x)$, so that we have $\nabla u=\rho^\frac{p-1+\beta}{p}\nabla w + \frac{p-1+\beta}{p}\rho^\frac{-1+\beta}{p}w\nabla \rho$.
		We need to estimate the term
		\begin{equation*}
			\begin{split}
				I:=\ & \I{\Om}|\nabla u(x)|_g^p\rho^\alpha dV_g\\
				&- \left(\frac{|p-1+\beta|}{p}\right)^p\I{\Om}\frac{|u(x)|^p|\nabla \rho(x)|_g^p}{|\rho(x)|^p}\rho^\alpha dV_g\\
				=\ & \I{\Om} \left(\left|\rho^\frac{p-1+\beta}{p}\nabla w + \left(\frac{p-1+\beta}{p}\right)\frac{w\nabla\rho}{\rho^\frac{1-\beta}{p}}  \right|^p_g\right. \\
				&\left. -\left(\frac{|p-1+\beta|}{p}\right)^p\left|\frac{w\nabla \rho}{\rho^\frac{1-\beta}{p}}\right|_g^p\right)\rho^\alpha.
			\end{split}
		\end{equation*}
		Let $p\geq 2$. Using \Cref{p-estimate} and \eqref{suphar-weak}, we get
		
		\begin{align}\label{eq1}
			\nonumber	
			I
			\geq\ & C(p)\I{\Om} \left|\rho^\frac{p-1+\beta}{p}\nabla w\right|^p\rho^\alpha
			+p\left(\frac{|p-1+\beta|}{p}\right)^{p-2}\left(\frac{p-1+\beta}{p}\right)\\
			\nonumber
			&\times\I{\Om}\left|\frac{w\nabla \rho}{\rho^\frac{1-\beta}{p}}\right|^{p-2}\left< \rho^\frac{p-1+\beta}{p}\nabla w,\frac{w\nabla \rho}{\rho^\frac{1-\beta}{p}} \right>\rho^\alpha\\
			\nonumber
			=\ & C(p)\I{\Om} \left|\nabla w\right|^p\rho^{p-1+\alpha+\beta}
			+\left(\frac{|p-1+\beta|}{p}\right)^{p-2}\left(\frac{p-1+\beta}{p}\right)\\
			\nonumber
			&\times\I{\Om}\left<\nabla |w|^p,|\nabla \rho|^{p-2}\nabla \rho \right>\rho^{\alpha+\beta}\\
			\geq\ & C(p)\I{\Om} \left|\nabla w\right|^p\rho^{p-1+\alpha+\beta}.
		\end{align}
		The last line in the above calculation follows from \eqref{suphar-weak}. Now, set $\psi=\frac{w}{\varphi}$. Then $\nabla w=\varphi \nabla \psi+\psi\nabla\varphi$. We have, from \Cref{p-estimate} and \eqref{ode-weak},
		\begin{align}\label{eq2}
			\nonumber	\I{\Om}&\rho^{p-1+\alpha+\beta}|\nabla w|_g^p\\
			\nonumber
			=\ & \I{\Om}\rho^{p-1+\alpha+\beta}|\varphi \nabla \psi+\psi\nabla\varphi|_g^p\\
			\nonumber	\geq\ & C(p)\I{\Om}\rho^{p-1+\alpha+\beta}|\varphi\nabla \psi|_g^p
			\\
			\nonumber
			&+\I{\Om}p\rho^{p-1+\alpha+\beta}|\psi\nabla \varphi|^{p-2}<\psi\nabla \varphi,\varphi\nabla\psi>
			+\I{\Om}\rho^{p-1+\alpha+\beta} |\psi\nabla \varphi|_g^p\\
			\nonumber
			\geq\ & \I{\Om}\left<\rho^{p-1+\alpha+\beta}\varphi|\nabla \varphi|^{p-2}\nabla \varphi,\nabla|\psi|^p\right>
			+\I{\Om}\rho^{p-1+\alpha+\beta} |\psi\nabla \varphi|_g^p\\
			\nonumber
			=\ & \I{\Om}\left<|\nabla \varphi|^{p-2}\nabla \varphi,\nabla(\varphi|\psi|^p)\right>\rho^{p-1+\alpha+\beta}\\
			\nonumber
			=\ & \I{\Om}\left<|\nabla \varphi|^{p-2}\nabla \varphi,\nabla(|u|^p\varphi^{-p+1}\rho^{-p+1-\beta})\right>\rho^{p-1+\alpha+\beta}\\
			\nonumber
			=\ & \I{\Om}\left<|\nabla \varphi|^{p-2}\nabla \varphi,\nabla(|u|^p)\right>\varphi^{-p+1}\rho^\alpha
			- (p-1) \I{\Om}|\nabla \varphi|^p|u|^p\rho^\alpha\varphi^{-p}\\
			\nonumber
			&- (p-1+\beta) \I{\Om}\left<|\nabla \varphi|^{p-2}\nabla \varphi,\nabla\rho\right>|u|^p\varphi^{-p+1}\rho^{-1+\alpha}\\
			\geq\ & \I{\Om} \frac{V}{C(p)}|u|^p \rho^\alpha.
		\end{align}
		Combining \eqref{eq1} and \eqref{eq2}, we get the desired result.
		
		Now let us consider the case $p< 2$. Using \Cref{p-estimate} and \eqref{suphar-weak}, we get
		\begin{align}\label{eq4}
			\nonumber	I\geq\ &
			\I{\Om} \frac{p(p-1)}{2}\frac{\left|\rho^\frac{p-1+\beta}{p}\nabla w\right|^2\rho^\alpha}{\left|\left|\frac{p-1+\beta}{p}\frac{w\nabla\rho}{\rho^\frac{1-\beta}{p}}\right|+\left|\rho^\frac{p-1+\beta}{p}\nabla w\right|\right|^{2-p}}\\
			\nonumber
			&+p\left|\frac{p-1+\beta}{p}\frac{w\nabla\rho}{\rho^\frac{1-\beta}{p}}\right|^{p-2}\left<\frac{p-1+\beta}{p}\frac{w\nabla\rho}{\rho^\frac{1-\beta}{p}},\rho^\frac{p-1+\beta}{p}\nabla w\right>\rho^\alpha\\
			\nonumber	=\ & \I{\Om} \frac{p(p-1)}{2}\frac{\rho^{1+\alpha+\beta}\left|\nabla w\right|^2}{\left|\left|\frac{p-1+\beta}{p}w\nabla\rho\right|+\left|\rho\nabla w\right|\right|^{2-p}}\\
			\nonumber
			&+\left(\frac{|p-1+\beta|}{p}\right)^{p-2}\left(\frac{p-1+\beta}{p}\right)\I{\Om} \left<|\nabla\rho|^{p-2}\nabla\rho,\nabla |w|^p\right>\rho^{\alpha+\beta}\\
			\geq\ & \frac{p^{3-p}(p-1)}{2}\I{\Om} \frac{\rho^{p-1+\alpha+\beta}|\nabla w|^2}{\left||(p-1+\beta)w\rho^{-1}\nabla\rho|+p|\nabla w|\right|^{2-p}}.
		\end{align}
		Now, we estimate the last term using H\"older's inequality, convexity of $t\mapsto t^p$ in the following:\pagebreak
		\begin{align}\label{eq3}
			\nonumber	\I{\Om}&\rho^{p-1+\alpha+\beta}|\nabla w|^p\\
			\nonumber	
			=\ & \I{\Om} \left(\frac{\rho^{p-1+\alpha+\beta}|\nabla w|^2}{\left||(p-1+\beta)w\rho^{-1}\nabla\rho|+p|\nabla w|\right|^{2-p}}\right)^\frac{p}{2}\\
			\nonumber	&\times \left(\rho^{\frac{(p-1+\alpha+\beta)(2-p)}{2}}\left||(p-1+\beta)w\rho^{-1}\nabla\rho|+p|\nabla w|\right|^\frac{2p-p^2}{2}\right)\\
			\nonumber	\leq\ & \left(\I{\Om} \left(\frac{\rho^{p-1+\alpha+\beta}|\nabla w|^2}{\left||(p-1+\beta)w\rho^{-1}\nabla\rho|+p|\nabla w|\right|^{2-p}}\right)\right)^\frac{p}{2}\\
			\nonumber	&\times \left(\I{\Om}\left(\rho^{\frac{(p-1+\alpha+\beta)(2-p)}{2}}\left||(p-1+\beta)w\rho^{-1}\nabla\rho|+p|\nabla w|\right|^\frac{2p-p^2}{2}\right)^\frac{2}{2-p}\right)^\frac{2-p}{2}\\
			\nonumber	=\ & \left(\I{\Om} \frac{\rho^{p-1+\alpha+\beta}|\nabla w|^2}{\left||(p-1+\beta)w\rho^{-1}\nabla\rho|+p|\nabla w|\right|^{2-p}}\right)^\frac{p}{2}\\
			\nonumber	&\times \left(\I{\Om}\rho^{p-1+\alpha+\beta}\left||(p-1+\beta)w\rho^{-1}\nabla\rho|+p|\nabla w|\right|^p\right)^\frac{2-p}{2}\\
			\nonumber	\leq\ & \left(\I{\Om} \frac{\rho^{p-1+\alpha+\beta}|\nabla w|^2}{\left||(p-1+\beta)w\rho^{-1}\nabla\rho|+p|\nabla w|\right|^{2-p}}\right)^\frac{p}{2}\\
			&\times \left(2^{p-1}|p-1+\beta|^p\I{\Om}\rho^{p-1+\alpha+\beta}|w\rho^{-1}\nabla\rho|^p\right.\\
			\nonumber
			&\left. +p^p2^{p-1}\I{\Om}\rho^{p-1+\alpha+\beta}|\nabla w|^p\right)^\frac{2-p}{2}.
		\end{align}
		Observing that since we have assumed \eqref{p<2}, we can apply \Cref{hardy1} to get
		\begin{align*}
			\I{\Om}\rho^{p-1+\alpha+\beta}|\nabla w|^p
			\leq& \ \left(\I{\Om} \frac{\rho^{p-1+\alpha+\beta}|\nabla w|^2}{\left||(p-1+\beta)w\rho^{-1}\nabla\rho|+p|\nabla w|\right|^{2-p}}\right)^\frac{p}{2}\\&
			\times \left(p^p2^p\I{\Om}\rho^{p-1+\alpha+\beta}|\nabla w|^p\right)^\frac{2-p}{2}.
		\end{align*}
		After setting $\psi=\frac{w}{\varphi}$, so that $\nabla w=\varphi \nabla \psi+\psi\nabla\varphi$ and then applying \Cref{p-estimate} and \eqref{ode-weak}, we get
		\begin{align*}
			\frac{p^{3-p}(p-1)}{2}&\I{\Om} \frac{\rho^{p-1+\alpha+\beta}|\nabla w|^2}{\left||(p-1+\beta)w\rho^{-1}\nabla\rho|+p|\nabla w|\right|^{2-p}}\\
			\geq\ & 2^{p-3}p(p-1)\I{\Om}\rho^{p-1+\alpha+\beta}|\nabla w|^p\\
			=\ &2^{p-3}p(p-1)\I{\Om}\rho^{p-1+\alpha+\beta}|\varphi \nabla \psi+\psi\nabla\varphi|^p\\
			\geq\ & 2^{p-3}p^2(p-1)\I{\Om} \rho^{p-1+\alpha+\beta} |\psi\nabla\varphi|^{p-2}\left<\psi\nabla\varphi,\varphi \nabla \psi \right>\\
			&+2^{p-3}p(p-1)\I{\Om}\rho^{p-1+\alpha+\beta}|\psi\nabla \varphi|^p\\
			=\ &2^{p-3}p(p-1)\I{\Om}  \left<\rho^{p-1+\alpha+\beta}\varphi|\nabla\varphi|^{p-2}\nabla\varphi, \nabla |\psi|^p \right>\\
			&+2^{p-3}p(p-1)\I{\Om}\rho^{p-1+\alpha+\beta}|\psi\nabla \varphi|^p\\	
			=\ &2^{p-3}p(p-1)\I{\Om}  \left<|\nabla\varphi|^{p-2}\nabla\varphi, \nabla (\varphi|\psi|^p) \right>\rho^{p-1+\alpha+\beta}\\
			\geq\ & \I{\Om} V|u|^p\rho^\alpha.
		\end{align*}
		The last line follows following the same equalities as in \eqref{eq2}.
		This with \eqref{eq4} proves the theorem.	
	\end{proof}
	\bigskip


\begin{thebibliography}{10}
		
		\bibitem{AbCoPe}
		B.~Abdellaoui, E.~Colorado, and I.~Peral,
		\newblock Some improved {C}affarelli-{K}ohn-{N}irenberg inequalities,
		\newblock {\em Calc. Var. Partial Differential Equations}, 23(3):327--345,
		2005.
		
		\bibitem{AdChRa}
		Adimurthi, N. Chaudhuri, and M. Ramaswamy,
		\newblock An improved {H}ardy-{S}obolev inequality and its application,
		\newblock {\em Proc. Amer. Math. Soc.}, 130(2):489--505, 2002.
		
		\bibitem{AdEs}
		Adimurthi and M. Esteban,
		\newblock An improved {H}ardy-{S}obolev inequality in {$W^{1,p}$} and its
		application to {S}chr\"{o}dinger operators,
		\newblock {\em NoDEA Nonlinear Differential Equations Appl.}, 12(2):243--263,
		2005.
		
		\bibitem{AdFiTe}
		Adimurthi, S. Filippas, and A. Tertikas,
		\newblock On the best constant of {H}ardy-{S}obolev inequalities,
		\newblock {\em Nonlinear Anal.}, 70(8):2826--2833, 2009.
		
		\bibitem{AdSe}
		Adimurthi and A. Sekar,
		\newblock Role of the fundamental solution in {H}ardy-{S}obolev-type
		inequalities,
		\newblock {\em Proc. Roy. Soc. Edinburgh Sect. A}, 136(6):1111--1130, 2006.
		
		\bibitem{AhKo}
		S. Ahmetolan and I. Kombe,
		\newblock Improved {H}ardy and {R}ellich type inequalities with two weight
		functions,
		\newblock {\em Math. Inequal. Appl.}, 21(3):885--896, 2018.
		
		\bibitem{BaFiTe}
		G.~Barbatis, S.~Filippas, and A.~Tertikas.
		\newblock A unified approach to improved {$L^p$} {H}ardy inequalities with best
		constants,
		\newblock {\em Trans. Amer. Math. Soc.}, 356(6):2169--2196, 2004.
		
		\bibitem{BrVa}
		H. Brezis and J. V\'{a}zquez,
		\newblock Blow-up solutions of some nonlinear elliptic problems,
		\newblock {\em Rev. Mat. Univ. Complut. Madrid}, 10(2):443--469, 1997.
		
		\bibitem{CaKrLa}
		C. Cazacu, D. Krejcirik, and A. Laptev,
		\newblock Hardy inequalities for magnetic $ p $-Laplacians,
		\newblock {\em arXiv preprint}  arXiv:2201.02482, 2022.
		
		\bibitem{Chau}
		N.~Chaudhuri,
		\newblock Bounds for the best constant in an improved {H}ardy-{S}obolev
		inequality,
		\newblock {\em Z. Anal. Anwendungen}, 22(4):757--765, 2003.
		
		\bibitem{Chavel}
		I. Chavel,
		\newblock {\em Riemannian geometry}, volume~98 of {\em Cambridge Studies in
			Advanced Mathematics}.
		\newblock Cambridge University Press, Cambridge, second edition, 2006.
		\newblock A modern introduction.
		
		\bibitem{CuPe}
		S. Cuomo and A. Perrotta,
		\newblock On best constants in {H}ardy inequalities with a remainder term,
		\newblock {\em Nonlinear Anal.}, 74(16):5784--5792, 2011.
		
		\bibitem{AmDi}
		L. D'Ambrosio and S. Dipierro,
		\newblock Hardy inequalities on {R}iemannian manifolds and applications,
		\newblock {\em Ann. Inst. H. Poincar\'{e} C Anal. Non Lin\'{e}aire},
		31(3):449--475, 2014.
		
		\bibitem{DeFrPi}
		B. Devyver, M. Fraas, and Y. Pinchover,
		\newblock Optimal {H}ardy weight for second-order elliptic operator: an answer
		to a problem of {A}gmon,
		\newblock {\em J. Funct. Anal.}, 266(7):4422--4489, 2014.
		
		\bibitem{DuLaPh}
		N. Duy, N. Lam, and L. Phi,
		\newblock Improved {H}ardy inequalities and weighted {H}ardy type inequalities
		with spherical derivatives,
		\newblock {\em Rev. Mat. Complut.}, 35(1):1--23, 2022.
		
		\bibitem{DuLaTrYi}
		N. Duy, N. Lam, N. Triet, and W. Yin,
		\newblock Improved {H}ardy inequalities with exact remainder terms,
		\newblock {\em Math. Inequal. Appl.}, 23(4):1205--1226, 2020.
		
		\bibitem{GaGrMi}
		F. Gazzola, H. Grunau, and E. Mitidieri,
		\newblock Hardy inequalities with optimal constants and remainder terms,
		\newblock {\em Trans. Amer. Math. Soc.}, 356(6):2149--2168, 2004.
		
		\bibitem{KoOz}
		I. Kombe and M. \"{O}zaydin,
		\newblock Improved {H}ardy and {R}ellich inequalities on {R}iemannian
		manifolds,
		\newblock {\em Trans. Amer. Math. Soc.}, 361(12):6191--6203, 2009.
		
		\bibitem{KoAb}
		I. Kombe and A. Yener,
		\newblock Weighted {H}ardy and {R}ellich type inequalities on {R}iemannian
		manifolds,
		\newblock {\em Math. Nachr.}, 289(8-9):994--1004, 2016.
		
		\bibitem{Kristaly}
		A. Krist\'{a}ly,
		\newblock Sharp uncertainty principles on {R}iemannian manifolds: the influence
		of curvature,
		\newblock {\em J. Math. Pures Appl. (9)}, 119:326--346, 2018.
		
		\bibitem{Lindq}
		P. Lindqvist,
		\newblock {\em Notes on the {$p$}-{L}aplace equation}, volume 102 of {\em
			Report. University of Jyv\"{a}skyl\"{a} Department of Mathematics and
			Statistics}.
		\newblock University of Jyv\"{a}skyl\"{a}, Jyv\"{a}skyl\"{a}, 2006.
		
		\bibitem{Mora}
		N. Ghoussoub, and A. Moradifam,
		\newblock {\em On the best possible remaining term in the {H}ardy inequality}, 
		\newblock  {\em Proc. Natl. Acad. Sci. USA}, 105(37):13746--13751, 2008.
		
		\bibitem{MeWaZh}
		C. Meng, H. Wang, and W. Zhao,
		\newblock Hardy type inequalities on closed manifolds via {R}icci curvature,
		\newblock {\em Proc. Roy. Soc. Edinburgh Sect. A}, 151(3):993--1020, 2021.
		
		\bibitem{Nguy}
		V. Nguyen,
		\newblock Improved critical {H}ardy inequality and {L}eray-{T}rudinger type
		inequalities in {C}arnot groups,
		\newblock {\em Ann. Fenn. Math.}, 47(1):121--138, 2022.
		
		\bibitem{TeZo}
		A.~Tertikas and N.~B. Zographopoulos,
		\newblock Best constants in the {H}ardy-{R}ellich inequalities and related
		improvements,
		\newblock {\em Adv. Math.}, 209(2):407--459, 2007.
		
		\bibitem{Thiam}
		E. Thiam,
		\newblock Weighted {H}ardy inequality on {R}iemannian manifolds,
		\newblock {\em Commun. Contemp. Math.}, 18(6):1550072, 25, 2016.
		
	\end{thebibliography}
	
\end{document}